\documentclass[12pt,reqno]{article}

\usepackage[top=2.5cm,bottom=2.5cm,left=2.5cm,right=2.5cm]{geometry}
\usepackage{graphicx}
\usepackage{latexsym}
\usepackage{amsmath,amssymb}
\usepackage{amscd}
\usepackage[arrow,matrix]{xy}
\usepackage{graphicx}
\usepackage{xcolor}
\usepackage{times}
\usepackage[]{subfigure}
 \usepackage{cite}
 \usepackage{float}

\usepackage{amscd,amsmath,amsthm,amssymb}

\theoremstyle{plain}
\newtheorem{theorem}{Theorem}[section]

\newtheorem{corollary}[theorem]{Corollary}

\theoremstyle{definition}

\newtheorem{definition}[theorem]{Definition}
\newtheorem{example}[theorem]{Example}

\begin{document}

\title{2-Ruled hypersurfaces in Minkowski 4-space and their constructions via octonions }
\author {{Ameth NDIAYE$^{1}$}\thanks{{
 E--mail: \texttt{ameth1.ndiaye@ucad.edu.sn} (A. NDIAYE)}},\texttt{ }Zehra \"{O}zdemir$^{2}$\footnote{{
 E--mail: \texttt{zehra.ozdemir@amasya.edu.tr} (Z. \"{O}zdemir)}} \\
 \begin{small}{$^{1}$D\'epartement de Math\'ematiques, FASTEF, UCAD, Dakar, Senegal.}\end{small}\\ 
\begin{small}{$^{2}$Department of Mathematics, Arts and Science Faculty, Amasya University, 05189 Amasya, Turkey.}\end{small}\\}
\date{}
\maketitle%


\begin{abstract} 
In this paper, we define three types of 2-Ruled hypersurfaces in the Minkowski 4-space $\mathbb{E}^4_1$. We obtain Gaussian and mean curvatures of the 2-ruled hypersurfaces of type-1 and type-2, and some characterizations about its minimality. We also deal with the first Laplace-Beltrami operators of these types of 2-Ruled hypersurfaces in $\mathbb{E}^4_1$. Moreover, the importance of this paper is that the definition of these surfaces by using the octonions in $\mathbb{E}^4_1$. Thus, this is a new idea and make the paper original. We give an example of 2-ruled hypersurface constructed by octonion and we visualize the projections of the images with MAPLE program. Furthermore, the optical fiber can be defined as a one-dimensional object embedded in the 4-dimensional Minkowski space $\mathbb{E}_{1}^{4}$. Thus, as a discussion, we investigate the geometric evolution of a linearly polarized light wave along an optical fiber by means of the 2-ruled hypersurfaces in a four-dimensional Minkowski space. 
\end{abstract}
\begin{small} {\textbf{MSC:} 53A10, 53C42, 53C50, 53Z05, 53B50, 37C10, 57R25.}
\end{small}\\
\begin{small} {\textbf{Keywords:} 2-Ruled hypersurfaces, curvature, Ruled surfaces, Vector fields, Electromagnetic theory, quaternion algebra, octonion algebra.} 
\end{small}\\
\maketitle

\section{Introduction}
The study of submanifolds of a given ambiant space is a naturel interesting problem which enriches our knowledge and understanding of the geometry of the space itself, see \cite{Berger, Carmo}. The theory of ruled surfaces in $\mathbb{R}^3$ is a classical subject in diffrential geometry and ruled hypersurfaces in higher dimensions have also been studied by many authors. For Ruled surfaces and their study one can see \cite{Dillen, Divjak, Flory, Guler}.\\
A 2-ruled hypersurface in $\mathbb{R}^4$ is a one-parameter family of planes in $\mathbb{R}^4$. This is a generalization of ruled surfaces in $\mathbb{R}^3$.\\
In \cite{Saji}, K. Saji study singularities of 2-ruled hypersurfaces in Euclidean 4-space. After defining a non-degenerate 2-ruled hypersurface he gives a necessary and sufficient condition for such a map germ to be right-left equivalent to the cross cap $\times$ interval. And he discusses the behavior of a generic 2-ruled hypersurface map.\\
In \cite{Mustapha} the authors obtain the Gauss map (unit normal vector field) of a 2-ruled hypersurface in Euclidean 4-space with the aid of its general parametric equation. They also obtain Gaussian and mean curvatures of the 2-ruled hypersurface and they give some characterizations about its minimality. Finally, they deal with the first and second Laplace-Beltrami operators of 2-ruled hypersurfaces in $\mathbb{E}^4$. In \cite{yayli,yayli1}, Aslan et al. characterize the ruled surface through quaternions in  $\mathbb{E}^3$ and $\mathbb{E}_{1}^3$. In three dimensions, the quaternions can be used to characterize the ruled surfaces. Identically, the 2-ruled hypersurfaces can be constructed by octonions, for more information about octonions see \cite{John}.\\
Motivated by the above two works, we study in this paper the 2-Ruled hypersurfaces in the Minkowski 4-space $\mathbb{E}^4_1$. We define three types of 2-Ruled hypersurfaces in $\mathbb{E}^4_1$ and we obtain Gaussian and mean curvatures of the 2-ruled hypersurface and some characterizations about its minimality. Moreover, we contract these surfaces via octonions in $\mathbb{E}^4_1$. We also deal with the first Laplace-Beltrami operators of these type of 2-Ruled hypersurfaces in $\mathbb{E}^4_1$. At the end, as an application, we investigate the geometric evolution of a linearly polarized light wave along an optical fiber by means of the 2-ruled hypersurfaces in a four-dimensional Minkowski space.

\section{Preliminaries}
Let $\mathbb{R}^4 = \lbrace (x_0, x_1, x_2, x_3) | x_i \in\mathbb{R} (i = 0, 1, 2, 3) \rbrace$ be an 4-dimensional cartesian space. For any $x = (x_0, x_1, x_2, x_3)$, $y =
(y_0, y_1, y_2, y_3) \in\mathbb{R}^4$, the pseudo-scalar product of $x$ and $y$ is defined by
\begin{eqnarray}\label{1}
\langle x, y\rangle=-x_0y_0+\sum_{i=1}^3 x_iy_i.
\end{eqnarray}
We call $(\mathbb{R}^4, \langle , \rangle)$  the Minkowski 4-space. We shall write $\mathbb{R}^4_1$ instead of $(\mathbb{R}^4, \langle , \rangle)$. We say that a non-zero vector $x\in\mathbb{R}^4_1$ is spacelike, lightlike or timelike if $\langle x , x\rangle>0$, $\langle x, x\rangle=0$ or $\langle x,x \rangle <0$  respectively. The norm of the vector $x\in\mathbb{R}^4_1$ is
\begin{eqnarray}\label{2}
\Vert x\Vert=\sqrt{\vert\langle x,x \rangle\vert }.
\end{eqnarray}
We now define the Hyperbolic 3-space by
\begin{eqnarray}\label{3}
 H^3_+(-1)=\lbrace x\in\mathbb{R}^4_1 \vert \langle x , x\rangle=-1, x_0>0\rbrace,
\end{eqnarray}
and the Sitter 3-space by
\begin{eqnarray}\label{4}
S^3_1=\lbrace x\in\mathbb{R}^4_1 \vert \langle x , x\rangle=1\rbrace.
\end{eqnarray}
We also define the light cone at the origin by
\begin{eqnarray}\label{5}
 LC=\lbrace x\in\mathbb{R}^4_1 \vert x_0\neq 0, \langle x , x\rangle=0\rbrace.
\end{eqnarray}
If $\overrightarrow{x}=(x_0,x_1,x_2, x_3) $, $\overrightarrow{y}=(y_0,y_1,y_2,y_3)$ and $\overrightarrow{z}=(z_0,z_1,z_2,z_3)$ are three vectors in $\mathbb{R}^4_1$, then vector product are defined by
  \begin{eqnarray}\label{eq-(1.2)}
\overrightarrow{x}\times\overrightarrow{y}\times\overrightarrow{z}=\det
\left[
    \begin{array}{cccc}
      -e_1 & e_2 & e_3 & e_4 \\
      x_0 & x_1 & x_2 & x_3 \\
      y_0 & y_1 & y_2 & y_3 \\
      z_0 & z_1 & z_2 & z_3
     \end{array}
\right].
  \end{eqnarray}
If
\begin{align*}
\varphi : \mathbb{R}^3&\longrightarrow \mathbb{R}^4_1 \cr
(x_0,x_1,x_2)&\longmapsto \varphi(x_0,x_1,x_2)=(\varphi_1(x_0,x_1,x_2),\varphi_2(x_0,x_1,x_2),\varphi_3(x_0,x_1,x_2),\varphi_4(u_1,u_2,u_3))
\end{align*}
is a hypersurface in Minkowski 4-space $\mathbb{R}^4_1$, then the Gauss map (i.e., the unit normal vector field), the matrix forms of the first and second fundamental forms are
\begin{align}\label{eq-(1.4)}
G=\frac{\varphi_{x_0}\times\varphi_{x_1}\times\varphi_{x_2}}{\Vert\varphi_{x_0}\times\varphi_{x_1}\times\varphi_{x_2}\Vert},
\end{align}
\begin{align}\label{eq-(1.5)}
[g_{ij}]=
\left[
   \begin{array}{ccc}
    g_{11} & g_{12} & g_{13} \\
    g_{21} & g_{22} & g_{23} \\
    g_{31} & g_{32} & g_{33}
   \end{array}
\right]
\end{align}
and
\begin{align}\label{eq-(1.6)}
[h_{ij}]=
\left[
  \begin{array}{ccc}
   h_{11} & h_{12} & h_{13} \\
   h_{21} & h_{22} & h_{23} \\
   h_{31} & h_{32} & h_{33}
  \end{array}
\right],
\end{align}
respectively, where the coefficients $g_{ij}=\langle \varphi_{x_i},\varphi_{x_j}\rangle$, $h_{ij}=\langle \varphi_{x_ix_j},G\rangle$, $\varphi_{x_i}=\frac{\partial \varphi(x_0,x_1,x_2)}{\partial x_i}$, $\varphi_{x_ix_i}=\frac{\partial^2\varphi(x_0,x_1,x_2)}{\partial x_ix_j}$, $i,j \in \{0,1,2\}$.\\

Also, the matrix of shape operator of the hypersurface $\varphi$ is
\begin{eqnarray}\label{eq-(1.7)}
S=[a_{ij}]=[g^{ij}]\cdot[h_{ij}],
\end{eqnarray}
where $[g^{ij}]$ is the inverse matrix of $[g_{ij}]$.\\
With aid of (\ref{eq-(1.5)})-(\ref{eq-(1.7)}), the Gaussian curvature and mean curvature of a hypersurface in $E^4$ are given by
\begin{eqnarray}\label{eq-(1.8)}
K=\frac{\det[h_{ij}]}{\det[g_{ij}]}
\end{eqnarray}
and
\begin{eqnarray}\label{eq-(1.9)}
3H=trace(S),
\end{eqnarray}
respectively (\ref{eq-(1.9)}).

Let the octonion parameterized by
\begin{eqnarray}
q=a_0+a_1e_1+a_2e_2+a_3e_3+a_4e_4+a_5e_5+a_6e_6+a_7e_7,
\end{eqnarray}
where $a_i, i=0,1,...,7$ are real numbers and the $e_i, i=0,1,...,7$ satisfy the following
\begin{itemize}
\item $e_1,...,e_7$ are square roots of $-1$,
\item $e_i$ and $e_j$ anticommute when $i\neq j$:
$$e_ie_j=-e_je_i$$
\item the index cycling identity holds:
$$e_ie_j=e_k\Rightarrow e_{i+1}e_{j+1}=e_{k+1}$$
where we think of the indices as living in $\mathbb{Z}_7$, and
\item the index doubling identity holds:
$$e_ie_j=e_k\Rightarrow e_{2i}e_{2j}=e_{2k}.$$
\end{itemize}
Now we assume that the reals $a_5=a_6=a_7=0$ and we get the expression
\begin{eqnarray}
Q=a_0+a_1e_1+a_2e_2+a_3e_3+a_4e_4,
\end{eqnarray}
called particular octonion.\\
This particular octonion can be also given in the form
\begin{eqnarray}
Q=S(Q) + V (Q),
\end{eqnarray}
where $S(Q) = a_0$ is the scalar and $V (Q)= a_1e_1+a_2e_2+a_3e_3+a_4e_4$ is the vector part of
$Q$. If $S(Q) = 0$, then $Q = a_1e_1+a_2e_2+a_3e_3+a_4e_4$ is called a pure particular octonion. Particular octonion product of any particular octonion $Q = S(Q) + V (Q)$ and $P = S(P) + V (P)$ is defined by
\begin{eqnarray}
Q\star P\star I=S(Q)S(P)-\langle V(Q), V(P)\rangle+S(q)V (p)\nonumber\\
+ S(p)V (q) + V (q)\times V (p)\times I,
\end{eqnarray}
where $\langle ,\rangle$ and $\times$ denote the usual scalar and vector products in $\mathbb{R}_{1}^4$, respectively, and $I$ is a unitary element of particular octonion.\\
Now we denote the set of all dual numbers by
\begin{eqnarray}
\mathbb{D}=\lbrace A=a+\varepsilon a^*/a,a^*\in\mathbb{R}\rbrace,
\end{eqnarray}
where $\varepsilon$ is the dual unit and satisfying
$$\varepsilon\neq 0, \,\,\varepsilon^2=0\,\,\,\,\,and\,\,\,\,\, r\varepsilon=\varepsilon r,\,\,\,\,\forall  r\in\mathbb{R}.$$
For any dual numbers $A=a+\varepsilon a^*$ and $B=b+\varepsilon b^*$, we have the addition and the multiplication expressed by
$$A+B=(a+b)+\varepsilon(a^*+b^*)$$
and
$$AB=ab+\varepsilon(a^*b+ab^*),$$
respectively.\\
Dual numbers form the module
\begin{eqnarray}
\mathbb{D}^4=\lbrace \tilde{A}=a+\varepsilon a^*/a,a^*\in\mathbb{R}^4\rbrace,
\end{eqnarray}
which is a commutative and associative ring.The element $\tilde{A}\in\mathbb{D}^4$ is called dual vector. The scalar and vector products of any dual vectors $\tilde{A}=a+\varepsilon a^*$ and $\tilde{B}=b+\varepsilon b^*$ are defined by
\begin{eqnarray}
\langle\title{A}, \title{B}\rangle_D=\langle a,b\rangle+\varepsilon(\langle a,b^*\rangle+\langle a^*,b\rangle)
\end{eqnarray}
and
\begin{eqnarray}
\tilde{A}\times_D\tilde{B}\times_D I=a\times b\times I+\varepsilon(a\times b^*\times I+a^*\times b\times I),
\end{eqnarray}
respectively. In the last two equalities, $\langle,\rangle$ and $\times$ denote the usual scalar
and vector products in $\mathbb{R}_{1}^4$, respectively. And the norm of a dual vector $\tilde{A}=a+\varepsilon a^*$ is defined to be
\begin{eqnarray}
N_{\tilde{A}}=\langle\title{A}, \title{A}\rangle_D=\vert a\vert^2+2\varepsilon\langle a,a^*\rangle\in\mathbb{D}.
\end{eqnarray}
Unit dual sphere is defined by
\begin{eqnarray}
\mathbb{S}^3_\mathbb{D}=\lbrace \tilde{A}=a+\varepsilon a^*/ \vert\tilde{A}\vert=1, \tilde{A}\in\mathbb{D}^4\rbrace.
\end{eqnarray}

\section{2-Ruled hypersurfaces of type-1 in $\mathbb{R}^4_1$}
A $2$-ruled hypersurface of type-1 in $\mathbb{R}^4_1$  means (the image of) a map
$\varphi:I_1\times I_2\times I_3\longrightarrow \mathbb{R}^4_1$ of the form
\begin{eqnarray}\label{eq-(2.1)}
\varphi(x,y,z)=\alpha(x)+y\beta(x)+z\gamma(x),
\end{eqnarray}
where $\alpha: I_1\longrightarrow \mathbb{R}^4_1$,  $\beta: I_2\longrightarrow S^3_1$, $\gamma :I_3\longrightarrow S^3_1$ are smooth maps, $S^3_1$ is the Sitter 3-space of $\mathbb{R}^4_1$ and $I_1, I_2, I_3$ are open intervals.\\
We call $\alpha$ a base curves $\beta$ and $\gamma$ director curves. The planes $ (y,z)\longrightarrow \alpha(x)+y\beta(x)+z\gamma(x)$ are called rulings.\\
\quad
So, if we take
\begin{eqnarray}\label{eq-(2.2)}
\left.
    \begin{array}{llllll}
    \alpha(x)&=(&\alpha_{1}(x),&\alpha_2(x),&\alpha_3(x),&\alpha_4(x))  \\
    \beta(x )&=(&\beta_{1} (x),&\beta_2 (x),& \beta_3(x),&\beta_4(x) )   \\
    \gamma(x)&=(&\gamma_{1}(x),&\gamma_2(x),&\gamma_3(x),&\gamma_4(x))
\end{array}
\right\rbrace,
\end{eqnarray}
in (\ref{eq-(2.1)}), then we can write the 2-ruled hypersurface of type-1 as
\begin{align}\label{eq-(2.3)}
\varphi(x,y,z)&= \alpha(x)+y\beta(x)+z\gamma(x)\cr
&=(\varphi_1(x,y,z),\varphi_2(x,y,z),\varphi_3(x,y,z),\varphi_4(x,y,z))\cr
&=\left(
\begin{array}{cc}
\alpha_1(x)+y\beta_1(x)+z\gamma_1(x), & \alpha_2(x)+y\beta_2(x)+z\gamma_2(x),\\
\alpha_3(x)+y\beta_3(x)+z\gamma_3(x), & \alpha_4(x)+y\beta_4(x)+z\gamma_4(x)
\end{array}
  \right).
\end{align}

We see that $\displaystyle\left(-\beta_1^2+\sum^{3}_{i=1}(\beta_i)^2\right)=\left(-\gamma_1^2+\sum^{3}_{i=1}(\gamma_i)^2\right)=1$ and we state $\alpha_i=\alpha_i(x)$, $\beta_i=\beta_i(x)$, $\gamma_i=\gamma_i(x)$, $\varphi_i=\varphi_i(x,y,z)$, $f'=\frac{\partial f(x)}{\partial x}$, $f''=\frac{\partial^2f(x)}{\partial x\partial x}$, $i\in \{1,2,3,4\}$ and $f\in \{\alpha,\beta,\gamma\}$.\\
We denote by
\begin{eqnarray}
E_{ij}=\gamma_i(\alpha'_j+y\beta'_j+z\gamma'_j)\\
F_{ij}=\beta_i(\alpha'_j+y\beta'_j+z\gamma'_j).
\end{eqnarray}

Now, let us prove the following theorem which contains the Gauss map of the 2-ruled hypersurface of type-1 (\ref{eq-(2.3)}).
\begin{theorem}\label{thm-(2)}
The Gauss map of the 2-ruled hypersurface of type-1 $(\ref{eq-(2.3)})$ is
\begin{eqnarray}\label{eq-(2.12)}
G(x,y,z)=\frac{G_1(x,y,z)e_1+G_2(x,y,z)e_2+G_3(x,y,z)e_3+G_4(x,y,z)e_4}{D},
\end{eqnarray}
where
\begin{eqnarray}\label{eq-(2.13)}
G_1(x,y,z)=\beta_2(E_{43}-E_{34})+\beta_3(E_{24}-E_{42})+\beta_4(E_{32}-E_{23})\nonumber\\
G_2(x,y,z)=\beta_1(E_{43}-E_{34})+\beta_3(E_{14}-E_{41})+\beta_4(E_{31}-E_{13})\nonumber\\
G_3(x,y,z)=\beta_1(E_{24}-E_{42})+\beta_2(E_{41}-E_{14})+\beta_4(E_{12}-E_{21})\nonumber\\
G_4(x,y,z)=\beta_1(E_{32}-E_{23})+\beta_2(E_{13}-E_{31})+\beta_3(E_{21}-E_{12})
\end{eqnarray}
and
\begin{eqnarray}\label{eq-(2.15)}
D=\sqrt{-G^2_1(x,y,z)+\sum^4_{i=2}G^2_i(x,y,z)}.
\end{eqnarray}

\end{theorem}
\begin{proof}
If we differentiate (\ref{eq-(2.3)}) we get
\begin{eqnarray*}
{\left\{ \begin{array}{ccc}
\varphi_x(x,y,z)&=&\big(\alpha'_1+y\beta'_1+z\gamma'_1,\alpha'_2+y\beta'_2+z\gamma'_2, \alpha'_3+y\beta'_3+z\gamma'_3, \alpha'_4+y\beta'_4+z\gamma'_4\big)\\
\varphi_y(x,y,z) &=& \big(\beta_1, \beta_2, \beta_3, \beta_4\big)\\
\varphi_z(x,y,z) &=& \big(\gamma_1, \gamma_2, \gamma_3, \gamma_4\big).
\end{array}\right.}
\end{eqnarray*}
By using the vector product in (\ref{eq-(1.2)}), we get
\begin{eqnarray*}
\varphi_x\times\varphi_y\times\varphi_z &=&\Big(\beta_2(E_{43}-E_{34})+\beta_3(E_{24}-E_{42})+\beta_4(E_{32}-E_{23})\Big)e_1\\
&&+\Big(\beta_1(E_{43}-E_{34})+\beta_3(E_{14}-E_{41})+\beta_4(E_{31}-E_{13}) \Big)e_2\\
&&+\Big( \beta_1(E_{24}-E_{42})+\beta_2(E_{41}-E_{14})+\beta_4(E_{12}-E_{21})\Big)e_3\\
&&+\Big( \beta_1(E_{32}-E_{23})+\beta_2(E_{13}-E_{31})+\beta_3(E_{21}-E_{12}) \Big)e_4
\end{eqnarray*}
Now using the unit normal vector formula in (\ref{eq-(1.4)}) we get the result.
\end{proof}
From (\ref{eq-(1.5)}) we obtain the matrix of the first fundamental form
\begin{align}\label{eq-(1.5(1))}
[g_{ij}]=
\left[
   \begin{array}{ccc}
    -(\alpha'_1+y\beta'_1+z\gamma'_1)^2+\sum_{i=2}^4(\alpha'_i+y\beta'_i+z\gamma'_i)^2 & -F_{11}+\sum_{i=2}^4F_{ii} & -E_{11}+\sum_{i=2}^4E_{ii} \\
  -F_{11}+\sum_{i=2}^4F_{ii} & 1 & -\beta_1\gamma_1+\sum_{i=2}^4\beta_i\gamma_i \\
   -E_{11}+\sum_{i=2}^4E_{ii} & -\beta_1\gamma_1+\sum_{i=2}^4\beta_i\gamma_i & 1
   \end{array}
\right].
\end{align}
And we obtain the inverse matrix $[g^{ij}]$ of $[g_{ij}]$ as

\begin{align}\label{eq-(1.5(2))}
[g^{ij}]=\frac{1}{\det[g_{ij}]}
\left[
   \begin{array}{ccc}
    1-e^2 & ce-b & be-c \\
  ce-b & a-c^2 &bc-ae \\
  be-c & bc-ae & a-b^2
   \end{array}
\right].
\end{align}
where
\begin{equation}\label{eq-(2.22)}
\left.
\begin{aligned}
a=&-(\alpha'_1+y\beta'_1+z\gamma'_1)^2+\sum_{i=2}^4(\alpha'_i+y\beta'_i+z\gamma'_i)^2,\cr
b=&-F_{11}+\sum_{i=2}^4F_{ii},\cr
c=& -E_{11}+\sum_{i=2}^4E_{ii},\cr
e=&-\beta_1\gamma_1+\sum_{i=2}^4\beta_i\gamma_i
\end{aligned}
\right\rbrace
\end{equation}
and\\
\begin{align}\label{eq-(2.23)}
\det[g_{ij}]=-b^2+2cbe-c^2-ae^2+a=D.
\end{align}
Furthermore, from (\ref{eq-(1.6)}), the matrix from of the second fundamental from of the 2-ruled hypersurface (\ref{eq-(2.3)} is obtained
by
\begin{equation}\label{eq-(2.24)}
[h_{ij}]=
\left[
\begin{array}{ccc}
h_{11}  &  h_{12}  &  h_{13}  \cr
h_{21}  & 0  &  0  \cr
h_{31}  &  0  &  0
\end{array}
\right],
\end{equation}
where
\begin{equation}\label{eq-(2.25)}
\left.
\begin{aligned}
h_{11}&=\displaystyle \frac{-G_1(\alpha_1''+y\beta_1''+z\gamma_1'')+\sum^{4}_{i=2}G_i(\alpha_i''+y\beta_i''+z\gamma_i'')}{\sqrt{-G^2_1(x,y,z)+\sum^3_{i=1}G^2_i(x,y,z)}},\cr
h_{12}&=h_{21}=\displaystyle\frac{-G_1\beta'_1+\sum_{i=2}^4 G_i\beta'_i}{\sqrt{-G^2_1(x,y,z)+\sum^3_{i=1}G^2_i(x,y,z)}},\cr
h_{13}&=h_{31}=\displaystyle\frac{-G_1\gamma'_1+\sum_{i=2}^4 G_i\gamma'_i}{\sqrt{-G^2_1(x,y,z)+\sum^3_{i=1}G^2_i(x,y,z)}}.
\end{aligned}
\right\rbrace
\end{equation}
We can see easily that the $\det[h_{ij}]=0$. \\
Then we can give the following theorem by using (\ref{eq-(1.8)})
\begin{theorem}
The 2-ruled hypersurfaces of type-1 defined in (\ref{eq-(2.3)}) is flat.
\end{theorem}
Now we will prove the following theorem about the mean curvature
\begin{theorem}
The 2-ruled hypersurfaces of type-1 defined in (\ref{eq-(2.3)}) is minimal in $\mathbb{R}^4_1$, if
\begin{eqnarray}\label{eq-2.26(1)}
(1-e^2)\left[-G_1(\alpha_1''+y\beta_1''+z\gamma_1'')+\sum^{4}_{i=2}G_i(\alpha_i''+y\beta_i''+z\gamma_i'')\right]\nonumber\\
+(ce-b)\left[-G_1\beta'_1+\sum_{i=2}^4 G_i\beta'_i\right]+(be-c)\left[-G_1\gamma'_1+\sum_{i=2}^4 G_i\gamma'_i\right]\nonumber\\
+(ce-b)\left[-G_1\beta'_1+\sum_{i=2}^4 G_i\beta'_i\right]
+(be-c)\left[-G_1\gamma'_1+\sum_{i=2}^4 G_i\gamma'_i\right]=0
\end{eqnarray}
\end{theorem}
\begin{proof}
By (\ref{eq-(1.7)}) the matrix of the shape operator is
\begin{align*}
S=
\left[
   \begin{array}{ccc}
    1-e^2 & ce-b & be-c \\
  ce-b & a-c^2 &bc-ae \\
  be-c & bc-ae & a-b^2
   \end{array}
\right]\left[
   \begin{array}{ccc}
    h_{11} & h_{12} & h_{13} \\
   h_{21} & 0 &0 \\
 h_{31} & 0 & 0
   \end{array}
\right].
\end{align*}
Then we get the coefficients of $S$ by
\begin{eqnarray*}
S_{11}&=&(1-e^2)h_{11}+(ce-b)h_{21}+(be-c)h_{31}\\
S_{22}&=&(ce-b)h_{12}\\
S_{33}&=& (be-c)h_{13}.
\end{eqnarray*}
And using (\ref{eq-(2.25)}) and (\ref{eq-(1.9)}) we see that the 2-ruled hypersurfaces is minimal if
\begin{eqnarray*}
S_{11}+S_{22}+S_{33}=0,
\end{eqnarray*}
then that end the proof.
\end{proof}
\begin{corollary}
If the curves $\beta$ and $\gamma$ are orthogonal then the 2-ruled hypersurfaces of type-1 defined in (\ref{eq-(2.3)}) is minimal if
\begin{eqnarray}\label{eq-2.27}
\left[-G_1(\alpha_1''+y\beta_1''+z\gamma_1'')+\sum^{4}_{i=2}G_i(\alpha_i''+y\beta_i''+z\gamma_i'')\right]\nonumber\\
-b\left[-G_1\beta'_1+\sum_{i=2}^4 G_i\beta'_i\right]-c\left[-G_1\gamma'_1+\sum_{i=2}^4 G_i\gamma'_i\right]\nonumber\\
-b\left[-G_1\beta'_1+\sum_{i=2}^4 G_i\beta'_i\right]
-c\left[-G_1\gamma'_1+\sum_{i=2}^4 G_i\gamma'_i\right]=0.
\end{eqnarray}
\end{corollary}
The Laplace-Beltrami operator of a smooth function $f=f(x^1, x^2, x^3)$ of class $C^3$ with respect to the first fundamental form of a hypersurface is defined as follows:
\begin{eqnarray}\label{eq-3.1}
\Delta f=\frac{1}{\sqrt{\det[g_{ij}]}}\sum_{i,j}^3\frac{\partial}{\partial x^i}\left(\sqrt{\det[g_{ij}]}g^{ij}\frac{\partial f}{\partial x^j}\right).
\end{eqnarray}
Using (\ref{eq-3.1}) we get the Laplace-Beltrami operator of the 2-ruled hypersurface of type-1 (\ref{eq-(2.2)}) by
\begin{eqnarray*}
\Delta\varphi=(\Delta\varphi_1,\Delta\varphi_2,\Delta\varphi_3,\Delta\varphi_4),
\end{eqnarray*}
where
\begin{eqnarray}\label{eq-3.2}
\Delta\varphi_i=\frac{1}{\sqrt{ D} }\left[
\begin{aligned}
&   \frac{\partial}{\partial x}\left(\frac{(1-e^2)\varphi_{ix}+(ce-b)\varphi_{iy}+(be-c)\varphi_{iz}}{\sqrt{\det[g_{ij}]}}\right)\\
& + \frac{\partial}{\partial y}\left(\frac{(ce-b)\varphi_{ix}+(a-c^2)\varphi_{iy}+(bc-ae)\varphi_{iz}}{\sqrt{\det[g_{ij}]}}\right) \\
&+ \frac{\partial}{\partial z}\left(\frac{(be-c)\varphi_{ix}+(bc-ae)\varphi_{iy}+(a-b^2)\varphi_{iz}}{\sqrt{\det[g_{ij}]}}\right)
\end{aligned}
\right].
\end{eqnarray}
That is
\begin{eqnarray}\label{eq-3.3}
\Delta\varphi_i=\frac{1}{\sqrt{ D} }\left[
\begin{aligned}
&   \frac{\partial}{\partial x}\left(\frac{(1-e^2)(\alpha'_i+y\beta'_i+z\gamma'_i)+(ce-b)\beta_i+(be-c)\gamma_i}{\sqrt{\det[g_{ij}]}}\right)\\
& + \frac{\partial}{\partial y}\left(\frac{(ce-b)(\alpha'_i+y\beta'_i+z\gamma'_i)+(a-c^2)\beta_i+(bc-ae)\gamma_i}{\sqrt{\det[g_{ij}]}}\right) \\
&+ \frac{\partial}{\partial z}\left(\frac{(be-c)(\alpha'_i+y\beta'_i+z\gamma'_i)+(bc-ae)\beta_i+(a-b^2)\gamma_i}{\sqrt{\det[g_{ij}]}}\right)
\end{aligned}
\right].
\end{eqnarray}
If we suppose that $\beta$ and $\gamma$ are orthogonal,then the Laplace-Beltrami operator of the 2-ruled hypersuface of type-1 is given by
\begin{eqnarray}\label{eq-3.4}
\Delta\varphi_i=\frac{1}{\sqrt{ a-b^2-c^2} }\left[
\begin{aligned}
&   \frac{\partial}{\partial x}\left(\frac{(\alpha'_i+y\beta'_i+z\gamma'_i)-b\beta_i-c\gamma_i}{\sqrt{a-b^2-c^2}}\right)\\
& + \frac{\partial}{\partial y}\left(\frac{-b(\alpha'_i+y\beta'_i+z\gamma'_i)+(a-c^2)\beta_i+bc\gamma_i}{\sqrt{a-b^2-c^2}}\right) \\
&+ \frac{\partial}{\partial z}\left(\frac{-c(\alpha'_i+y\beta'_i+z\gamma'_i)+bc\beta_i+(a-b^2)\gamma_i}{\sqrt{a-b^2-c^2}}\right)
\end{aligned}
\right].
\end{eqnarray}
\begin{theorem}
The components of the Laplace-Beltrami operator of the 2-ruled hypersurface of type-1 are
\begin{eqnarray}\label{eq-3.5(2)}
\Delta\varphi_i=\frac{1}{\sqrt{Q} }\left[
\begin{aligned}
&\frac{(\alpha''_i+y\beta''_i+z\gamma''_i)-(b\beta_i)_x-(c\gamma_i)_x)Q-P_1(\alpha'_i+y\beta'_i+z\gamma'_i-b\beta_i-c\gamma_i )}{Q^{\frac{3}{2}}}\\
& +\frac{(-b\beta'_i+((a-c^2)\beta_i)_y+(bc\gamma_i)_y)Q-P_2(-b(\alpha'_i+y\beta'_i+z\gamma'_i)+(a-c^2)\beta_i+bc\gamma_i)}{Q^\frac{3}{2}} \\
&+ \frac{(-c\gamma'_i+(bc\beta_i)_z+((a-b^2)\gamma_i)_z)Q-P_3(-c(\alpha'_i+y\beta'_i+z\gamma'_i)+bc\beta_i+(a-b^2)\gamma_i)}{Q^\frac{3}{2}}
\end{aligned}
\right],
\end{eqnarray}
where $i = 1, 2, 3, 4$; $\beta$ and $\gamma$ are orthogonal; $Q=a-b^2-c^2$, $P_1=a_x-2bb_x-2cc_x$, $P_2=a_y-2bb_y-2cc_y$, $P_3=a_z-2bb_z-2cc_z$.
\begin{example}
Let $\varphi$ be the 2-ruled hypersurface of type-1 defined by
\begin{eqnarray*}
\varphi(x,y,z)=\Big(3x+7+\frac{y}{\sqrt{7}}, -5x+1+\frac{z}{\sqrt{5}}, x+\frac{2y\sqrt{2}}{\sqrt{7}}, -4x-1+\frac{2z}{\sqrt{5}}\Big).
\end{eqnarray*}
We take
$\alpha(x)=(3x+7, -5x+1, x; -4x-1)$, $\beta(x)=(\frac{1}{\sqrt{7}}, 0, \frac{2\sqrt{2}}{\sqrt{7}}, 0)$, $\gamma(x)=(0,\frac{1}{\sqrt{5}},0, \frac{2}{\sqrt{5}})$.\\
An easy computation show that $\varphi$ is minimal. And the Laplace-Beltrami operator of $\varphi$ is zero.
\end{example}
\end{theorem}
\section{2-Ruled hypersurfaces of type-2 in $\mathbb{R}^4_1$}
A $2$-ruled hypersurface of type-1 in $\mathbb{R}^4_1$  means (the image of) a map
$\varphi:I_1\times I_2\times I_3\longrightarrow \mathbb{R}^4_1$ of the form
\begin{eqnarray}\label{eq-(4.1)}
\varphi(x,y,z)=\alpha(x)+y\beta(x)+z\gamma(x),
\end{eqnarray}
where $\alpha: I_1\longrightarrow \mathbb{R}^4_1$,  $\beta: I_2\longrightarrow H^3_+(-1)$, $\gamma :I_3\longrightarrow H^3_+(-1)$ are smooth maps, $H^3_+(-1)$ is the hyperbolic 3-space of $\mathbb{R}^4_1$ and $I_1, I_2, I_3$ are open intervals.\\
We call $\alpha$ a base curve, $\beta$ and $\gamma$ director curves. The planes $ (y,z)\longrightarrow \alpha(x)+y\beta(x)+z\gamma(x)$ are called rulings.\\
\quad
So, if we take
\begin{eqnarray}\label{eq-(4.2)}
\left.
    \begin{array}{llllll}
    \alpha(x)&=(&\alpha_{1}(x),&\alpha_2(x),&\alpha_3(x),&\alpha_4(x))  \\
    \beta(x )&=(&\beta_{1} (x),&\beta_2 (x),& \beta_3(x),&\beta_4(x) )   \\
    \gamma(x)&=(&\gamma_{1}(x),&\gamma_2(x),&\gamma_3(x),&\gamma_4(x))
\end{array}
\right\rbrace
\end{eqnarray}
in (\ref{eq-(2.1)}), then we can write the 2-ruled hypersurface of type-1 as
\begin{align}\label{eq-(4.3)}
\varphi(x,y,z)&= \alpha(x)+y\beta(x)+z\gamma(x)\cr
&=(\varphi_1(x,y,z),\varphi_2(x,y,z),\varphi_3(x,y,z),\varphi_4(x,y,z))\cr
&=\left(
\begin{array}{cc}
\alpha_1(x)+y\beta_1(x)+z\gamma_1(x), & \alpha_2(x)+y\beta_2(x)+z\gamma_2(x),\\
\alpha_3(x)+y\beta_3(x)+z\gamma_3(x), & \alpha_4(x)+y\beta_4(x)+z\gamma_4(x)
\end{array}
  \right).
\end{align}

We see that $\displaystyle\left(-\beta_1^2+\sum^{3}_{i=1}(\beta_i)^2\right)=\left(-\gamma_1^2+\sum^{3}_{i=1}(\gamma_i)^2\right)=-1$ and we state $\alpha_i=\alpha_i(x)$, $\beta_i=\beta_i(x)$, $\gamma_i=\gamma_i(x)$, $\varphi_i=\varphi_i(x,y,z)$, $f'=\frac{\partial f(x)}{\partial x}$, $f''=\frac{\partial^2f(x)}{\partial x\partial x}$, $i\in \{1,2,3,4\}$ and $f\in \{\alpha,\beta,\gamma\}$.\\
From (\ref{eq-(1.5(1))}) we obtain the matrix of the first fundamental form
\begin{align}\label{eq-(4.4)}
[g_{ij}]=
\left[
   \begin{array}{ccc}
    -(\alpha'_1+y\beta'_1+z\gamma'_1)^2+\sum_{i=2}^4(\alpha'_i+y\beta'_i+z\gamma'_i)^2 & -F_{11}+\sum_{i=2}^4F_{ii} & -E_{11}+\sum_{i=2}^4E_{ii} \\
  -F_{11}+\sum_{i=2}^4F_{ii} & -1 & -\beta_1\gamma_1+\sum_{i=2}^4\beta_i\gamma_i \\
   -E_{11}+\sum_{i=2}^4E_{ii} & -\beta_1\gamma_1+\sum_{i=2}^4\beta_i\gamma_i & -1
   \end{array}
\right].
\end{align}
And we obtain the inverse matrix $[g^{ij}]$ of $[g_{ij}]$ as

\begin{align}\label{eq-(4.5)}
[g^{ij}]=\frac{1}{\det[g_{ij}]}
\left[
   \begin{array}{ccc}
    1-e^2 & ce+b & be+c \\
  ce+b & -a-c^2 &bc-ae \\
  be+c & bc-ae & -a-b^2
   \end{array}
\right].
\end{align}
where
$a$, $b$, $c$ and $e$ are the same in (\ref{eq-(2.22)})
and
\begin{align}\label{eq-(4.6)}
\det[g_{ij}]=b^2+2cbe+c^2-ae^2+a=D.
\end{align}
Furthermore, from (\ref{eq-(1.6)}), the matrix from of the second fundamental from of the 2-ruled hypersurface (\ref{eq-(4.3)}) is the same given in (\ref{eq-(2.24)}) and (\ref{eq-(2.25)}). And we have the following theorem since the $\det [h_{ij}]=0$.
\begin{theorem}
The 2-ruled hypersurfaces of type-2 defined in (\ref{eq-(4.3)}) is flat.
\end{theorem}
For the mean curvature we have
\begin{theorem}
The 2-ruled hypersurfaces of type-2 defined in (\ref{eq-(4.3)}) is minimal in $\mathbb{R}^4_1$, if
\begin{eqnarray}\label{eq-2.26}
(1-e^2)\left[-G_1(\alpha_1''+y\beta_1''+z\gamma_1'')+\sum^{4}_{i=2}G_i(\alpha_i''+y\beta_i''+z\gamma_i'')\right]\nonumber\\
+(ce+b)\left[-G_1\beta'_1+\sum_{i=2}^4 G_i\beta'_i\right]+(be+c)\left[-G_1\gamma'_1+\sum_{i=2}^4 G_i\gamma'_i\right]\nonumber\\
+(ce+b)\left[-G_1\beta'_1+\sum_{i=2}^4 G_i\beta'_i\right]
+(be+c)\left[-G_1\gamma'_1+\sum_{i=2}^4 G_i\gamma'_i\right]=0.
\end{eqnarray}
\end{theorem}
\begin{proof}
By (\ref{eq-(1.7)}) the matrix of the shape operator is
\begin{align*}
S=
\left[
   \begin{array}{ccc}
    1-e^2 & ce-b & be-c \\
  ce-b & a-c^2 &bc-ae \\
  be-c & bc-ae & a-b^2
   \end{array}
\right]\left[
   \begin{array}{ccc}
    h_{11} & h_{12} & h_{13} \\
   h_{21} & 0 &0 \\
 h_{31} & 0 & 0
   \end{array}
\right]
\end{align*}
Then we get the coefficients of $S$ by
\begin{eqnarray*}
S_{11}&=&(1-e^2)h_{11}+(ce+b)h_{21}+(be+c)h_{31}\\
S_{22}&=&(ce+b)h_{12}\\
S_{33}&=& (be+c)h_{13}.
\end{eqnarray*}
And using (\ref{eq-(2.25)}) and (\ref{eq-(1.9)}) we see that the 2-ruled hypersurfaces of type-2 is minimal if
\begin{eqnarray*}
S_{11}+S_{22}+S_{33}=0,
\end{eqnarray*}
then that end the proof.
\end{proof}
\begin{corollary}
If the curves $\beta$ and $\gamma$ are orthogonal then the 2-ruled hypersurfaces of type-2 defined in (\ref{eq-(4.3)}) is minimal if
\begin{eqnarray}\label{eq-2.27(1)}
\left[-G_1(\alpha_1''+y\beta_1''+z\gamma_1'')+\sum^{4}_{i=2}G_i(\alpha_i''+y\beta_i''+z\gamma_i'')\right]\nonumber\\
+b\left[-G_1\beta'_1+\sum_{i=2}^4 G_i\beta'_i\right]+c\left[-G_1\gamma'_1+\sum_{i=2}^4 G_i\gamma'_i\right]\nonumber\\
+b\left[-G_1\beta'_1+\sum_{i=2}^4 G_i\beta'_i\right]
+c\left[-G_1\gamma'_1+\sum_{i=2}^4 G_i\gamma'_i\right]=0.
\end{eqnarray}
\end{corollary}
To end this section, we will give the operator of Laplace-Beltrami in the following theorem
\begin{theorem}
The components of the Laplace-Beltrami operator of the 2-ruled hypersurface of type-2 are
\begin{eqnarray} \label{eq-3.5}
\Delta\varphi_i=\frac{1}{\sqrt{Q} }\left[
\begin{aligned}
&\frac{(\alpha''_i+y\beta''_i+z\gamma''_i)+(b\beta_i)_x+(c\gamma_i)_x)Q-P_1(\alpha'_i+y\beta'_i+z\gamma'_i+b\beta_i+c\gamma_i )}{Q^{\frac{3}{2}}}\\
& +\frac{(b\beta'_i+((-a-c^2)\beta_i)_y+(bc\gamma_i)_y)Q-P_2(b(\alpha'_i+y\beta'_i+z\gamma'_i)+(-a-c^2)\beta_i+bc\gamma_i)}{Q^\frac{3}{2}} \\
&+ \frac{(c\gamma'_i+(bc\beta_i)_z+((-a-b^2)\gamma_i)_z)Q-P_3(c(\alpha'_i+y\beta'_i+z\gamma'_i)+bc\beta_i+(-a-b^2)\gamma_i)}{Q^\frac{3}{2}}
\end{aligned}
\right],
\end{eqnarray}
where $i = 1, 2, 3, 4$; $\beta$ and $\gamma$ are orthogonal; $Q=a+b^2+c^2$, $P_1=a_x+2bb_x+2cc_x$, $P_2=a_y+2bb_y+2cc_y$, $P_3=a_z+2bb_z+2cc_z$.
\end{theorem}
\begin{example}\label{E1}
Let $\varphi$ be the 2-ruled hypersurface of type-2 defined by
\begin{eqnarray*}
\varphi(x,y,z)=\Big(\frac{x^4}{4}-\frac{2y}{\sqrt{3}}+\sqrt{2}, 2x+1+\frac{z}{\sqrt{7}}, -3x+\frac{y}{\sqrt{3}}, \frac{x^3}{3}+\frac{z\sqrt{6}}{\sqrt{7}}\Big).
\end{eqnarray*}
An easy computation show that $\varphi$ is minimal. And the Laplace-Beltrami operator of $\varphi$ is zero.
\end{example}

\section{2-ruled hypersurfaces constructed by particular octonions}
Now we give the definition of the 2-ruled hypersurface constructed by the particular octonion.
\begin{definition}
Let $\tilde{\gamma}=a(t)+\varepsilon a^*(t)$ and $\tilde{\beta}=b(t)+\varepsilon b^*(t)$ be two curves
on the unit dual sphere $\mathbb{S}^3_\mathbb{D}$, the 2-ruled hypersurfaces corresponding to these
curves is
\begin{eqnarray}
\varphi(r,s,t)=\alpha(t)+sa(t)+rb(t),
\end{eqnarray}
where $\alpha(t)=a(t)\times a^*(t)\times I+b(t)\times b^*(t)\times I$.
\end{definition}
Let $u(t)$ be a curve in $\mathbb{R}^4$. We can define two particular octonions
\begin{eqnarray*}
Q(s,t)=s+u(t), \,\,\,\,\,P(r,t)=r+u(t),
\end{eqnarray*}
where $S(Q(s,t))=s$, $S(P(s,t))=r$ and $V(Q(s,t))=V(P(s,t))=u(t)$.
\begin{theorem}
Let $v(t)$ and $w(t)$ be two curve on unit sphere in $\mathbb{S}^3_\mathbb{D}$ and let their position vectors
be perpendicular to the position vector of the curve $u(t)$ (i.e $\vert v(t)\vert=\vert w(t)\vert=1$ and $\langle v(t), u(t)\rangle=\langle w(t), u(t)\rangle=0$. Then the sum defined by
\begin{eqnarray}
\varphi(s,r,t)=\alpha(t)+sw(t)+rv(t),
\end{eqnarray}
where $\alpha(t)=u(t)\times v(t)\times I+u(t)\times w(t)\times I$,
is a 2-ruled hypersurface constructed by the two particular octonions.
\end{theorem}
\begin{proof}
Since $S(Q(s,t))=s$, $S(P(s,t))=r$ and $V(Q(s,t))=V(P(s,t))=u(t)$, using the octonion product operator we have
\begin{eqnarray}\label{p1}
Q(s,t)\star w(t)\star I&=&(s + u(t))\star w(t)\star I\nonumber\\
&=&-\langle u(t),w(t)\rangle + sw(t) + u(t)\times w(t)\times I\nonumber\\
&=& sw(t) + u(t)\times w(t)\times I,
\end{eqnarray}
and the same calculus give
\begin{eqnarray}\label{p2}
Q(s,t)\star v(t)\star I&=&(s + u(t))\star v(t)\star I\nonumber\\
&=&rv(t) + u(t)\times v(t)\times I.
\end{eqnarray}
If we put (\ref{p1})+(\ref{p2}), we get
\begin{eqnarray*}
\varphi(s,r,t)&=&u(t)\times v(t)\times I+u(t)\times w(t)\times I+sw(t)+rv(t)\\
&=&
\alpha(t)+sw(t)+rv(t),
\end{eqnarray*}
where $\alpha(t)=u(t)\times v(t)\times I+u(t)\times w(t)\times I$.
\end{proof}
\begin{corollary}
Let $\tilde{\gamma}_1=a+\varepsilon a^*$ and $\tilde{\gamma}_2=b+\varepsilon b^*$ be dual number in $\mathbb{S}^3_\mathbb{D}$. Then, the particular octonion $\varphi(s,r,t)$ can be written as follows
\begin{eqnarray}
\varphi(s,r,t)=\alpha(t)+sa(t)+rb(t),
\end{eqnarray}
where $\alpha(t)=a(t)\times {a}^{\ast}(t)\times I+b(t)\times {b}^{\ast}(t)\times I$,
is a 2-ruled hypersurface constructed by the two particular octonions.
\end{corollary}
\begin{proof}
Let $\tilde{\gamma}_1=a+\varepsilon a^*$ and $\tilde{\gamma}_2=b+\varepsilon b^*$ be dual number in $\mathbb{S}^3_\mathbb{D}$. We know that 
$$Q(s,t)=s+a^*(t)$$
and 
$$P(r,t)=r+b^*(t)$$
are two particular octonions where $S(Q(s,t))=s$, $S(P(r,t))=r$,  $V(Q(s,t))= a^*(t)$ and $V(P(r,t))=b^*(t)$. So using the octonion product we have 
\begin{eqnarray*}
a(t)\star Q(s,t)\star I&=&a(t)(s+a^*(t))\nonumber\\
&=&-\langle a(t), a^*(t)\rangle+sa(t)+a(t)\times a^*(t)\times I.
\end{eqnarray*}
Since $\vert\tilde{\gamma}_1\vert=1$ we have $\langle a(t), a^*(t)\rangle=0$.
Then 
\begin{eqnarray}
a(t)\star Q(s,t)\star I=a(t)\times a^*(t)\times I+sa(t).\label{p3}
\end{eqnarray}
The same calculus gives also
\begin{eqnarray}
b(t)\star P(r,t)\star I=b(t)\times b^*(t)\times I+rb(t).\label{p4}
\end{eqnarray}
If we take (\ref{p3})+(\ref{p4}) and denote by $\varphi(s,r,t)=a(t)\star Q(s,t)\star I+b(t)\star P(r,t)\star I$, we get
\begin{eqnarray*}
\varphi(s,r,t)=\alpha(t)+sa(t)+rb(t).
\end{eqnarray*}
\end{proof}
\begin{example}\label{Ex3}
Let us take the particular octonions $Q(s,t)=s+u(t)$ and  $P(r,t)=r+u(t)$ defined by $u(t)=(-\cos t\cos 2 t,\cos t\sin 2 t, 0, 0)\in \mathbb{R}^{4}_1$. Then, we can find 
\begin{align*}
    w(t)=(\sin t\sin 2 t,\sin t\cos 2 t, \cos t, \sin t)\textit{ and }v(t)=(\cos t\sin 2 t,\sin t\sin 2 t, \sin t, -\cos t).
\end{align*}
 Thus, we can compute
 \begin{align*}
     \alpha(t)&=u(t)\times w(t)\times I+=u(t)\times v(t)\times I\\
     &=(0,0,\sin 2 t(\frac{1}{2}\sin 2t-\cos^{2} t),\sin 2 t(\frac{1}{2}\sin 2t+\cos^{2} t)).
 \end{align*}
Then, we reach the following 2-ruled hypersurface of type-1, 
\begin{displaymath}
 \varphi(s,t,r)=\left(
\begin{array}{c}
s \sin t\sin 2t+r\cos t\sin 2 t \\
s \sin t\cos 2t+r\sin t\sin 2 t  \\
\sin 2 t(\frac{1}{2}\sin 2t-\cos^{2} t)+s\cos t+r\sin t\\
\sin 2 t(\frac{1}{2}\sin 2t+\cos^{2} t)+s\sin t+r\cos t
\end{array}\right).
\end{displaymath}
\end{example}
Next, the image of the projections of 2-ruled hypersurface of type-1 in Example \ref{Ex3} onto  $\mathbb{R}_{1}^{3}$ constructed by particular octonion are visualized in Figure \ref{F1}. 
\begin{figure}[H]
\begin{center}
\subfigure[]{\includegraphics[width=3.3in]{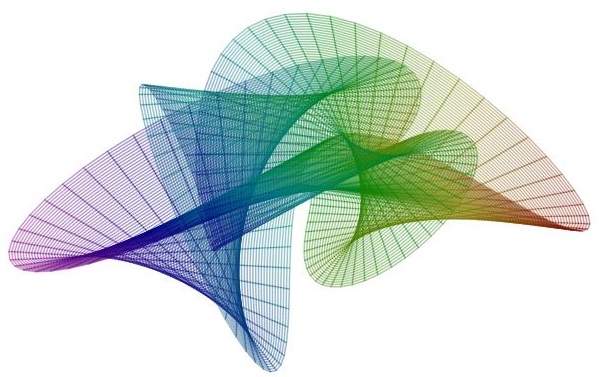}}
\subfigure[]{\includegraphics[width=3.1in]{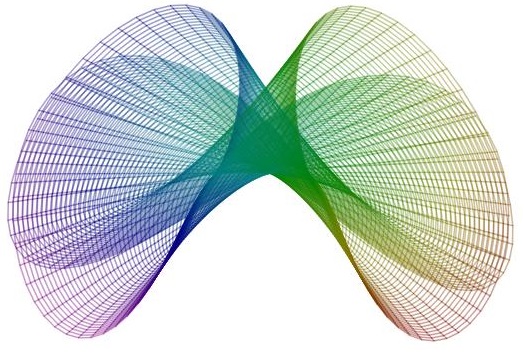}}
\subfigure[]{\includegraphics[width=3.3in]{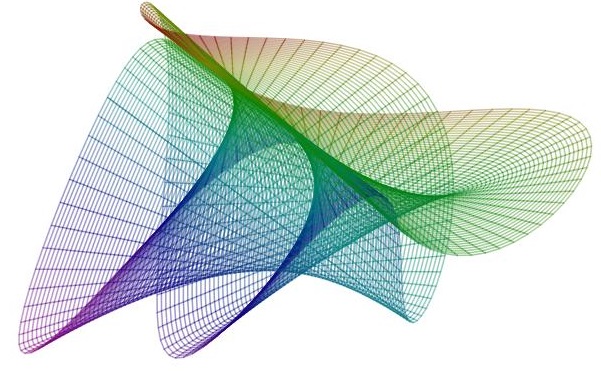}}
\subfigure[]{\includegraphics[width=3.1in]{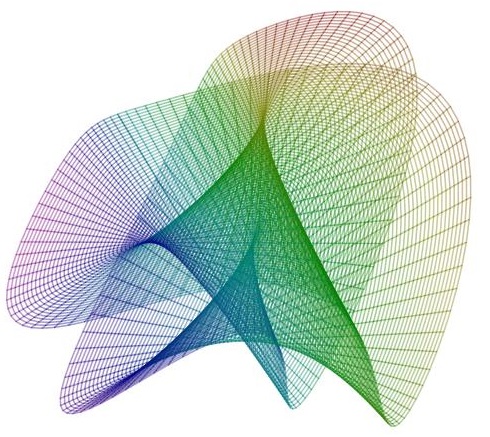}}
\end{center}
\caption{Some projections of 2-ruled hypersurface of type-1 constructed by particular octonion in $\mathbb{R}_{1}^{4}$}
\label{F1}
\end{figure}

\section{Some discussions related to the electromagnetic theory}
By identifying an optical fiber with a curve, we can give a geometric interpretation of the motion of a linearly polarized light wave through Frenet roof elements. As the linearly polarized light wave moves along the optical fiber, the Polarization plane rotates, and the image of the polarization vector (electric field) in the plane is a linear line.\\
Therefore, we can use ruled surfaces to model this movement geometrically. In particular, it would be very advantageous to use ruled surface equations instead of standard calculations when expressing the motion of a linearly polarized light wave along the optical fiber in 4 dimensions.\\
In this study, we defined three types of 2-ruled hypersurfaces in 4-dimensional Minkowski space $\mathbb{R}_{1}^{4}$. In this section we will give an interpretation of the motion of the polarized light wave in the 4-dimensional Minkowski space of these surfaces and give some motivated examples and visualize them through MAPLE program.\\
We demonstrate that the evolution of a linearly polarized light wave is associated with the movement of the parameter curve, which is the line segment in the formation of the ruled surface. If we match the parameter curve, which is the line segment of the ruled surface, with the polarization vector, the optical fiber as the other parameter curve is matched. Hence, the polarization vector moves in parallel along an optical fiber. This allows us to interpret the movement of the polarization vector (electric field) along an optical fiber geometrically in 4-dimensional space. 
\section{Conclusions}
In this paper, we gave the definition of three types of 2-ruled hypersurfaces and we calculated the mean curvature, the Gauss curvature and the Laplace-Bertrami operator of the two types of 2-ruled hypersurfaces. After, we constructed those 2-ruled hypersurfaces by using the particular octonion. In this construction, we gave an example and we visualized the images with MAPLE program. This construction is new and original. Then, we presented some discussions related to the 2-ruled hypersurfaces and the electromagnetic theory. For perspective, one can do the same in also Riemannian 4-manifolds and pseudo-Riemannian 4-manifolds.

\end{document}